\documentclass[11pt,a4paper]{article}
\usepackage{a4wide}
\usepackage{amssymb,amsfonts,amsmath,amsthm}

\usepackage{graphicx}
\usepackage[hidelinks]{hyperref}
\usepackage{mathtools}
\usepackage{xparse}
\usepackage{amssymb,amsfonts,amsmath}
\usepackage{graphicx,color,changebar}

\newcommand{\eg}{e.\,g.,\ }

\newcommand{\Hn}{ {\mathbb{H}_n} }   
\newcommand{\Sn}{ {\mathbb{S}_n} }   

\usepackage{todonotes}

\newtheorem{remark}{Remark}[section]

\newtheorem{theorem}{Theorem}[section]
\newtheorem{lemma}{Lemma}[section]
\newtheorem{definition}{Definition}[section]

\newcommand{\T}{{\mathcal T}}
\newcommand{\E}{\mathcal{E}}
\newcommand{\A}{{\mathcal A}}
\renewcommand{\S}{{\mathcal S}}

\DeclareMathOperator{\vect}{vec}

\DeclareMathOperator{\diag}{diag}

\DeclareMathOperator{\In}{\mathsf{Inertia}}
\DeclareMathOperator{\rank}{rank}

\newcommand{\Hnpsd}{ {\mathbb{H}_n^{\raisebox{0.15em}{{\fontsize{3}{2}\selectfont $\geq$}}}} } 
\newcommand{\Hnpd}{ {\mathbb{H}_n^{\raisebox{0.2em}{{\fontsize{3}{2}\selectfont $>$}}}} }  

\newcommand{\ie}{i.\,e.\ }

\usepackage{xstring}
\newcommand{\mat}[3][C]{
	\mathbb{#1}^{
		\IfSubStr{#2}{+}{(#2)}{#2}
		\times
		\IfSubStr{#3}{+}{(#3)}{#3}
	}
}
\DeclareDocumentCommand{\matz}{m m O{K} O{z}}{\mathbb{#3}[{#4}]^
	{
		\IfSubStr{#1}{+}{(#1)}{#1}
		\times
		\IfSubStr{#2}{+}{(#2)}{#2}
	}}

\title{Structured backward errors for eigenvalues of linear port-Hamiltonian descriptor systems}

\author{Volker Mehrmann\footnotemark[1], Paul Van Dooren\footnotemark[2]
}
\date{\today}
\begin{document}
\maketitle

\begin{abstract} When computing the eigenstructure of matrix pencils associated with the passivity analysis of perturbed port-Hamiltonian descriptor system using a structured generalized eigenvalue method, one should make sure that the computed spectrum satisfies the symmetries that corresponds to this structure and the underlying phy\-si\-cal system. We perform a backward error analysis and show that for matrix pencils associated with port-Hamiltonian descriptor systems and a given computed eigenstructure with the correct symmetry structure there always exists a nearby port-Hamiltonian descriptor system with exactly that eigenstructure. We also derive bounds for how near this system is and show that the stability radius of the system plays a role in that bound.  \hfill 
\end{abstract}

\footnotetext[1]{
Institut f\"ur Mathematik MA 4-5, TU Berlin, Str.\ des 17.\ Juni 136,
D-10623 Berlin, FRG.
\url{mehrmann@math.tu-berlin.de}. Supported by Deutsche Forschungsgemeinschaft via Project A02 within CRC 910 `Control of self-organized nonlinear systems' and priority program SPP 1984 through the project 'Computational Strategies for Distributed Stability Control
in Next-Generation Hybrid Energy Systems
'.}
\footnotetext[2]{
Department of Mathematical Engineering, Universit\'e catholique de Louvain, Louvain-La-Neuve, Belgium.
\url{paul.vandooren@uclouvain.be}. Supported by the {\it Deutsche Forschungsgemeinschaft},
through CRC 910 `Control of self-organized nonlinear systems'.
}

\section{Introduction} \label{sec:intro}
We study the perturbation analysis of the eigenstructure (finite and infinite eigenvalues, left and right eigenvectors) of matrix pencils associated the passivity analysis of linear time-invariant \emph{descriptor systems (generalized state-space systems)}  of the form
\begin{equation} \label{gstatespace}
 \begin{array}{rcl} E \dot x(t) & = & Ax(t) + B u(t),\ x(0)=0,\\
y(t)&=& Cx(t)+Du(t),
\end{array}
\end{equation}
where $u:\mathbb R\to\mathbb{C}^m$,   $x:\mathbb R\to \mathbb{C}^n$,  and  $y:\mathbb R\to\mathbb{C}^m$  are vector-valued functions denoting, respectively, the \emph{input}, \emph{state},
and \emph{output} of the system. Denoting real and complex $n$-vectors ($n\times m$ matrices) by $\mathbb R^n$, $\mathbb C^{n}$ ($\mathbb R^{n \times m}$, $\mathbb{C}^{n \times m}$), respectively, the coefficient matrices satisfy $A,E \in \mathbb{C}^{n \times n}$,   $B\in \mathbb{C}^{n \times m}$, $C\in \mathbb{C}^{m \times n}$, and  $D\in \mathbb{C}^{m \times m}$.  Note that we require that input and output dimensions are both equal to $m$; and that $sE-A$ is a  \emph{square regular pencil} $sE-A$, \ie $\det (sE-A)$ does not vanish identically for all $s\in \mathbb C$.

We will particularly focus on systems that are  \emph{positive real or passive} and their \emph{port-Hamiltonian realizations} (see next section). Our work is motivated by two applications, the first is the perturbation analysis arising from computational methods to compute the eigenstructure \cite{GolV96,Meh91,Van79,Van81b} and the second arises from the need to obtain small perturbations that bring the system back to this structure when it has been destroyed in the process of discretization,  model reduction, or other  computational techniques, \cite{AlaBKMM10,Bru11,BruS12,FreJ04,FreJV07,Gri03,Gri04,GusS01,OveV05}. In both applications the eigenstructure of an originally  passive system is perturbed due to  perturbations in the process. And then one either wants to determine a nearby passive system with the perturbed eigenstructure (if this exists) or one wants to perturb the eigenstructure so that it is that of a nearby passive system \cite{GilS18}.  A similar problem arises in stability analysis and the computation of stability radii and smallest pertubations that make a system stable \cite{GilMS18,GilS17,MadO95}. While most of these mentioned previous works are for standard passive  systems, here we deal with descriptor systems, as they arise from the linearization around stationary solutions of general systems of differenential-algebraic equations \cite{BieCM12,Dai89,KunM06,Meh91}.

Throughout this article we will use the following notation. The Hermitian (or conjugate) transpose (transpose) of a vector or matrix $V$ is denoted by
$V^{\mathsf{H}}$ ($V^{\mathsf{T}}$) and the identity matrix is denoted by $I_n$ or $I$ if the dimension is clear.
We denote the set of Hermitian and skew-Hermitian matrices in $\mathbb{C}^{n \times n}$, respectively, by $\Hn$ and $\Sn$.
Positive definiteness (semi-definiteness) of  $A\in \Hn$ is denoted by $A>0$ ($A\geq 0$). The set of all positive definite (positive semidefinite) matrices in $\Hn$ is denoted $\Hnpd$ ($\Hnpsd$). With $\In(H)$ of a Hermitian matrix $H$ we denote the triple of integers $\{p,n,z\}$ of numbers of positive, negative and zero eigenvalues of $H$.
The real and imaginary parts of a complex matrix $Z$ are written as $\Re (Z)$ and $\Im (Z)$, respectively, and $\imath$ is the imaginary unit.
The 2-norm of a matrix $M$ will be denoted by $\|M\|_2$ and the Frobenius norm by $\|M\|_F$. The Frobenius norm of a list of matrices $M_i,i=1,\ldots,k$ is defined as $\|(M_1, \ldots,M_k)\|_F:=\sqrt{\sum_{i=1}^k \|M_i\|_F^2}$.

The eigenstructure of matrix pencils is characterized by the \emph{Kronecker canonical form}.
\begin{theorem}\label{th:kcf}
Let $E,A\in {\mathbb C}^{n,m}$. Then there exist nonsingular matrices
$S\in {\mathbb C}^{n,n}$ and $T\in {\mathbb C}^{m,m}$ such that
\begin{equation}\label{kcf}
S(\lambda E-A)T=\diag({\cal L}_{\epsilon_1},\ldots,{\cal L}_{\epsilon_p},
{\cal L}^\top_{\eta_1},\ldots,{\cal L}^\top_{\eta_q},
{\cal J}_{\rho_1}^{\lambda_1},\ldots,{\cal J}_{\rho_r}^{\lambda_r},{\cal N}_{\sigma_1},\ldots,
{\cal N}_{\sigma_s}),
\end{equation}
where the block entries have the following properties:
\begin{enumerate}
\item[\rm (i)]
Every entry ${\cal L}_{\epsilon_j}$ is a bidiagonal block of size
${\epsilon_j}\times ({\epsilon_j+1})$, $\epsilon_j\in{\mathbb N}_0$,
of the form
\[
\lambda\left[\begin{array}{cccc}
1&0\\&\ddots&\ddots\\&&1&0
\end{array}\right]-\left[\begin{array}{cccc}
0&1\\&\ddots&\ddots\\&&0&1
\end{array}\right].
\]
\item[\rm (ii)]
Every entry ${\cal L}^\top_{\eta_j}$ is a bidiagonal block of size
$({\eta_j+1})\times {\eta_j}$, $\eta_j\in{\mathbb N}_0$,
of the form
\[
\lambda\left[\begin{array}{ccc}
1\\0&\ddots\\&\ddots&1\\&&0
\end{array}\right]-
\left[\begin{array}{ccc}
0\\1&\ddots\\&\ddots&0\\&&1
\end{array}\right].
\]
\item[\rm (iii)]
Every entry ${\cal J}_{\rho_j}^{\lambda_j}$ is a Jordan block of size
${\rho_j}\times{\rho_j}$, $\rho_j\in{\mathbb N}$, $\lambda_j\in{\mathbb C}$,
of the form
\[
\lambda\left[\begin{array}{cccc}
1\\&\ddots\\&&\ddots\\&&&1
\end{array}\right]-
\left[\begin{array}{cccc}
\lambda_j&1\\&\ddots&\ddots\\&&\ddots&1\\&&&\lambda_j
\end{array}\right].
\]
\item[\rm (iv)]
Every entry ${\cal N}_{\sigma_j}$ is a nilpotent block of size
${\sigma_j}\times {\sigma_j}$, $\sigma_j\in{\mathbb N}$,
of the form
\[
\lambda\left[\begin{array}{cccc}
0&1\\&\ddots&\ddots\\&&\ddots&1\\&&&0
\end{array}\right]-
\left[\begin{array}{cccc}
1\\&\ddots\\&&\ddots\\&&&1
\end{array}\right].
\]
\end{enumerate}
The Kronecker canonical form is unique up to permutation of the blocks.
\end{theorem}
%
A value $\lambda_0\in\mathbb C$ is called a (finite) generalized eigenvalue of $\lambda E-A$ if
$\operatorname{rank}(\lambda_0E-A)<\max_{\alpha\in\mathbb C}
\operatorname{rank}(\alpha E-A)$, and $\lambda_0=\infty$ is said to be an eigenvalue of $\lambda E-A$ if zero is an eigenvalue of $\lambda A-E$.
The blocks $\mathcal J_{\rho_j}$ as in (iii) are associated with the
finite eigenvalues of $\lambda E-A$, and the blocks $\mathcal N_{\sigma_j}$ as in (iv) correspond to the
eigenvalue $\infty$.
The size of the largest block ${\cal N}_{\sigma_j}$ is
called the \emph{index} $\nu$ of the pencil $\lambda E-A$, where, by convention,  $\nu=0$ if $E$ is invertible.
The matrix pencil $\lambda E-A$ is called \emph{regular} if $n=m$ and
$\operatorname{det}(\lambda_0 E-A)\neq 0$ for some $\lambda_0 \in \mathbb C$,
otherwise it is called \emph{singular}.

\subsection{Positive-realness, passivity, and port-Hamiltonian systems }\label{sec:prelim}
 By applying the Laplace transform to \eqref{gstatespace} and eliminating the state, we obtain the  \emph{transfer function}
\begin{equation} \label{EABCD}
\mathcal T(s):=D+C(sE-A)^{-1}B, 
\end{equation}
mapping the Laplace transform of $u$ to the Laplace transform of $y$.
On the imaginary axis $\imath \mathbb R$, $\mathcal T(\imath\omega)$
describes the \emph{frequency response} of the system. We have the following extensions of the concepts of positive realness and passivity to descriptor systems, see \eg \cite{Bru11,FreJ04}.
\begin{definition}\label{def:posreal}~
\begin{enumerate}
\item A transfer function $\T(s)$ as in \eqref{EABCD} is {\em positive real}
if it is i) analytic in the open right half complex plane (including $\infty$), and
ii) $\Phi(s):= \T(s)+[\T(s)]^{\mathsf{H}}\geq 0$ for all $s$ in the closed right half  complex plane. Moreover,
$\mathcal T(s)$ is \emph{strictly positive real} if $\Phi(s)>0$ for all  $s$ in the closed right half  complex plane.
\item
A system of the form (\ref{gstatespace}) is \emph{passive} if there exists a state-dependent
\emph{storage function}, $\mathcal H(x)\geq 0$, such that for any $t_1>t_0\in \mathbb R$
 the  \emph{dissipation inequality}
\begin{equation} \label{supply} \mathcal H(x(t_1))-\mathcal H(x(t_0)) \le \int_{t_0}^{t_1} \Re (y(t)^{\mathsf{H}}u(t)) \, dt
\end{equation}
holds.
If for all $t_1>t_0$, inequality \eqref{supply}
is strict then the system is called \emph{strictly passive}.
\end{enumerate}
\end{definition}
It is well-known, see \eg \cite{BeaMV19,FreJ04}, that a system with regular   pencil $sE-A$ that is controllable ($\rank [\,s E-A, \ B\,] =n$ for all $s\in \mathbb C$), and \emph{observable} ($\rank [\, s E^\mathsf{H}-A^\mathsf{H},\ C^\mathsf{H}\,] =n $ for all $s\in \mathbb C$) is passive if and only if it is positive real
and \emph{stable} (all finite eigenvalues of $sE-A$ are in the closed left half complex plane, and those on the imaginary axis including $\infty$ are \emph{semisimple}), and it is strictly passive if and only if it is strictly positive real and \emph{asymptotically stable} (all finite eigenvalues of $sE-A$ are in the open left half complex plane, and the infinite eigenvalues are \emph{semisimple}).

In recent years, the  special class of port-Hamiltonian (pH) realizations of passive systems has received a lot attention.
PH systems are a tool for energy-based modeling, see \cite{SchJ14}; with the energy storage function
$\mathcal{H}(x)=\frac12x^{\mathsf{H}}Ex$, the dissipation inequality (\ref{supply}) holds and so pH systems are always passive.  The (robust) representation of passive systems as pH systems has been analyzed  in \cite{BeaMV19,BeaMX15_ppt}, and in the extension to pH descriptor systems in \cite{BeaMXZ18,MehM19,Sch13,SchM18}.
\begin{definition}\label{def:ph}
A linear time-invariant \emph{port-Hamiltonian (pH) descriptor system} has the generalized state-space form
\begin{equation} \label{pH}
 \begin{array}{rcl} E \dot x  & = & (J-R) x + (G-P) u,\\
y&=& (G+P)^{\mathsf{H}}x+(S-N)u,
\end{array}
\end{equation}
where the coefficient matrices satisfy the symmetry conditions
\begin{equation}\label{symmcond}
\mathcal V:= \left[ \begin{array}{cccc} J & G \\ -G^{\mathsf{H}} & N \end{array} \right]=-\mathcal V^{\mathsf{H}},\
\mathcal W:= \left[ \begin{array}{cccc} R & P \\ P^{\mathsf{H}} & S \end{array} \right] =\mathcal W^{\mathsf{H}}\geq 0, \  E=E^{\mathsf{H}} \ge 0.
\end{equation}
\end{definition}
The correspondence with the generalized state-space realization \eqref{gstatespace} is given via
 $A=J-R$, $B=G-P$, $C^\mathsf{H}=G+P$, and $D=S-N$.

 The pH representation seems to be a very robust representation \cite{MehMS16,MehMS17}, it allows easy ways for structure preserving model reduction \cite{BeaG11,GugPBS12,PolS10} and it greatly simplifies optimization methods for computing stability and passivity radii \cite{GilMS18,GilS17,GilS18,OveV05}.

\subsection{Eigenstructure computation}\label{sec:eigen}
To analyze whether a system is passive, one can compute the eigenstructure of the \emph{para-Hermitian matrix function (even matrix pencil)} ${\mathcal T}(s)+{\mathcal T}^\mathsf{H}(-s)$, where ${\mathcal T}(s):=C(sE-A)^{-1}B+D$, which is given by
\begin{equation} \label{system}
{\mathcal S}(s) := s \left[ \begin{array}{ccc} 0 & E & 0 \\ -E^\mathsf{H} & 0 & 0 \\ 0 & 0 & 0 \end{array} \right]-
\left[ \begin{array}{ccc} 0 & A & B \\ A^{\mathsf{H}} & 0 & C^\mathsf{H} \\ B^\mathsf{H} & C & D^\mathsf{H}+D \end{array} \right].
\end{equation}
An often more advantageous representation of this pencil (in the context of pH systems) is obtained by applying a congruence transformation.
Consider the unitary matrix
\[
X:= \frac{1}{\sqrt{2}}\left[ \begin{array}{ccc} I_n & I_n \\ I_n & -I_n  \end{array} \right],
\]
and $\hat X := \diag(X, \frac{1}{\sqrt{2}} I_m)$,
then one can form the specially structured even pencil
\begin{equation} \label{systemeq}
\hat \S(s) := \hat X^\mathsf{H} {\mathcal S}(s) \hat X =  s \left[ \begin{array}{ccc} 0 & E & 0 \\ -E^\mathsf{H} & 0 & 0 \\ 0 & 0 & 0 \end{array} \right]-
 \left[ \begin{array}{ccc} -R & -J & G \\ -J^{\mathsf{H}} & R & -P \\ G^\mathsf{H} & -P^\mathsf{H} & S \end{array} \right].
\end{equation}
The system is passive if the pencil (\ref{system}) (or equivalently the pencil (\ref{systemeq})) is regular,  has no purely imaginary eigenvalues and is of index at most one, see \cite{FreJ04}, so computing the eigenvalues and the structure at $\infty$ allows to check passivity.

In view of this fact it is important to understand the perturbation  theory and the backward error analysis for the pencils (\ref{system}) (or equivalently the pencil (\ref{systemeq})). In this respect, the advantage of the form \eqref{systemeq} is that perturbations  can be mapped back directly to the data matrices $\{E,J,R,G,P,S\}$, while in (\ref{system}) this holds for the data matrices $\{E,A,B,C,D\}$. In both cases, for the backward error analysis, we  should also make sure that an arbitrary perturbation of the pencil can be mapped back in a \emph{structured} sense to perturbations in the data matrices, meaning (i) that the zero blocks should not be perturbed, (ii) that the perturbed matrices $E$, $\mathcal W$ should remain Hermitian, positive semidefinite, and $J$ skew-Hermitian, and (iii) that the repeated block entries  should have repeated perturbations as well.

There exist simple and well-conditioned transformations to go back and forth between the two representations~\eqref{gstatespace} and~\eqref{pH}, since
\[ \left[ \begin{array}{ccc} - R & G & S \\ J & -P & -N \end{array} \right]=  \frac{1}{\sqrt{2}} X
\left[ \begin{array}{ccc} A & B & D \\ A^\mathsf{H} & C^\mathsf{H} &  D^\mathsf{H} \end{array} \right],
\]
and
\[
  \left[ \begin{array}{ccc} A & B & D \\ A^\mathsf{H} & C^\mathsf{H} &  D^\mathsf{H} \end{array} \right] =  {\sqrt{2}}
   X^\mathsf{H}   \left[ \begin{array}{ccc} - R & G & S \\ J & -P & -N \end{array} \right].
\]
Thus for a perturbation analysis we can use either of the two system matrices.
We will consider the perturbations of the system pencil
\eqref{system} and then show how to extend them to the pencil \eqref{systemeq}.

\subsection{Backward error analysis}\label{sec: backward}
Let us assume that we have determined (via a computational method) an approximate eigenstructure  of the pencil ${\mathcal S}(s):=s{\mathcal E}-{\mathcal A}$. A backward error analysis yields that this eigenstructure  corresponds to the exact eigenstructure of a perturbed pencil
\[
({\mathcal S} +\Delta_{\mathcal S})(s) :=s({\mathcal E}+\Delta_{\mathcal E})-({\mathcal A}+\Delta_{\mathcal A}),
\]
where  $\|(\Delta_{\mathcal E},\Delta_{\mathcal A})\|_F \approx \epsilon \|({\mathcal E},{\mathcal A})\|_F$ and $\epsilon$ is the perturbation level. If the eigenstructure is determined by a backward stable algorithm, then $\epsilon$ is a small multiple of  the machine precision (round-off unit), but in other approximations it may be much larger, \eg when the perturbation arises from model reduction or other approximations.)
But even if the relative perturbation  $(\Delta_{\mathcal E},\Delta_{\mathcal A})$
is small, it is likely to destroy the structure present in the original pencil $s{\mathcal E}-{\mathcal A}$.

In view of this, in this paper, we study the following  questions.
\begin{enumerate}
	\item Does the perturbed (computed) eigenstructure correspond to that of a pencil with the same block and symmetry structure, \ie that of a port-Hamiltonian descriptor system.
	\item If the answer to the first question is positive, then what is the nearest  port-Hamiltonian descriptor system that has exactly this eigenstructure?
\item If the answer to the first question is negative, then what is the nearest  port-Hamiltonian descriptor system.
\end{enumerate}
Related questions have already been studied  in \cite{BanMNV20,BeaMV19,GilMS18,GilS18,MadO95,MehV19,MehV20,OveV05} in the context of finding best pH representations of stable and passive systems and the computation of stability and passivity radii of linear time-invariant dynamical systems. However, all these papers mainly deal with the classical port-Hamiltonian systems, \ie the case $E=I$; here we study pH descriptor
systems, which have extra properties that need to be incorporated \cite{MehMW18,MehMW20}.

\subsection{Stability Radii}\label{sec:stab}
A lower bound for the backward errors  that one can expect is the \emph{stability radius} of the generalized eigenvalue problem $sE-A$, since pH systems are guaranteed to be stable. The stability radius $\rho(E,A)$ of a pencil $sE-A$ is defined as the smallest perturbation $\| (\Delta_E,\Delta_A) \|_F$ that causes $s(E+\Delta_E)-(A+\Delta_A)$ to be on the border of the stability region \cite{GenV99}. In the descriptor case this happens when an eigenvalue reaches the imaginary axis, when the system has an index $\geq 0$ or when the pencil becomes singular \cite{DuLM13}.

In general, to characterize the smallest perturbation that makes a pencil singular is an open problem for unstructured descriptor systems \cite{ByeHM98,GugLM17,MehMW15} and requires very  complex optimization methods even in special cases.
However, for pH descriptor systems it has recently been shown in \cite{MehMW20} that these distances are easily characterized. Actually the distance to singularity is given by the smallest perturbation that generates a common nullspace of $E,J,R$, while actually the distance to instability and the structured distance to the nearest problem with an index $\nu \geq 2$ are the same and are characterized by the smallest perturbation that generates a common nullspace of $E$ and $R$ under structure preservering perturbations.

The classical stability radius is given by
	\begin{equation}
	\label{stabrad}
	 \rho(E,A)=\inf_{\|(\Delta_E,\Delta_A)\|_F} \{\Lambda(E+\Delta_E,A+\Delta_A)\cap\imath\Re \neq \emptyset\} =
	\inf_\omega \sigma_n(A-\imath \omega E)/\sqrt{1+\omega^2},
	\end{equation}
and the minimizing perturbation can be constructed from the $n$-th singular value triple $(\sigma_n,u_n,v_n)$ at the minimizing frequency $\omega$,
\begin{equation}
	\label{stabpert} \Delta_E := \imath\omega\sigma_n u_nv_n^\mathsf{H}/(1+\omega^2), \quad \Delta_A:=\sigma_n u_nv_n^\mathsf{H}/(1+\omega^2).
\end{equation}
For large scale system with pH structure, recently  a computational method to compute the stability radius has been derived in \cite{AliMM20}.

The paper is organized as follows. In the next Section \ref{sec:necsuf}, we construct a congruence transformation that restores the special structure of the pencil $s\E-\A$ and we compute upper bounds for its departure from the identity. In Section \ref{sec:numerical} we illustrate the results of Section \ref{sec:necsuf} with a number of numerical experiments. In Section \ref{sec:conclusion} we end with a few concluding remarks.

\section{Computing structured perturbation matrices that realize backward errors} \label{sec:necsuf}
In this section we address the first question whether an eigenstructure associated with a system of the form \eqref{gstatespace} corresponds to that of a pencil associated with a pH descriptor system. Assume that $sE-A$ is a regular pencil and  that
i) $\rank [\,s E-A, \ B\,]=n$ for all $s\in \mathbb C$, \ie the system is \emph{controllable}, and ii)  $\rank [\,s E^\mathsf{H}-A^\mathsf{H}, \ C^\mathsf{H}\,]=n $ for all $s\in \mathbb C$, \ie the system is \emph{observable}, see \cite{BunBMN99,Meh91} for a detailed discussion.
Before we characterize structured backward errors we need the following lemma.
\begin{lemma}\label{lem:lem1}
Consider a controllable and observable descriptor system of the form~\eqref{gstatespace} associated with a strictly passive pH descriptor system of the form \eqref{pH} and with $E$ positive definite.
If $s\E-\A$ is a regular pencil, then the finite generalized eigenvalues of $s\E-\A$ are symmetric with respect to the imaginary axis and there are exactly $m$ semisimple infinite generalized eigenvalues. Moreover,
\[
\In (\imath\E)=\{n,n,m\}, \quad \In(\A-\imath \omega \E) = \{n+m,n,0\},\; {\mathrm for\, all}\; \omega\in \mathbb R.
\]
\end{lemma}
\begin{proof}
Since $E>0$ and since $s\E-\A$ is regular, the pencil $s\E-\A$ has exactly $m$ infinite  eigenvalues. Since $\mathcal W>0$ by the assumption of strict passivity, it follows that $D^\mathsf{H}+D>0$ and hence the index is one and the finite eigenvalues of $s\E-\A$ are the eigenvalues of  the \emph{Hamiltonian matrix}
\[
 {\mathcal H} := \left[ \begin{array}{cc} E^{-1}A & 0 \\ 0 & -A^\mathsf{H}E^{-\mathsf{H}}
 \end{array}\right]- \left[ \begin{array}{cc} E^{-1}B \\ -C^\mathsf{H} \end{array} \right](D^\mathsf{H}+D)^{-1}\left[ \begin{array}{cc} C  & B^\mathsf{H}E^{-\mathsf{H}} \end{array} \right]
\]
 obtained by forming the Schur complement of $s\E-\A$ with respect to the block $D^\mathsf{H}+D>0$. It is well-known, see
 \cite{Meh91,PaiV81} that Hamiltonian matrices have a spectrum that is symmetric with respect to the imaginary axis.
 The inertia of the Hermitian matrix $(\imath\E)$ is clearly $\{n,n,m\}$, since $E$ is invertible. Since we have assumed controllability and observability, it is also well-known \cite{Meh91,Van81b} that $s\E-\A$ has no purely imaginary eigenvalues.
%
 %
 %
%
\end{proof}
A similar result as Lemma~\ref{lem:lem1} can also be obtained for the case that $E$ and/or $\mathcal W$ are only semidefinite. In this case one has to separate the differential and the algebraic equations and one has to make the stronger assumption that the pencil $s\E-\A$ is of index one. This can be achieved  via structured staircase forms, see e.g. \cite{BunBMN99} for general descriptor system and \cite{BeaGM20} for pH descriptor systems.
In the following we treat the case discussed in Lemma~\ref{lem:lem1}, \ie we assume that $E$ and $\mathcal W$ are positive definite so that we are not on the boundary of the set of passive systems.

When we perturb the pencil $s\E-\A$, it is clear that we cannot allow for arbitrary perturbations. The symmetry of the finite spectrum follows from the fact that $\A$ is Hermitian end $\E$ is skew Hermitian. We will therefore require that the perturbation preserves this, and hence that
the backward errors $\Delta_{\mathcal A}$ and $\Delta_{\mathcal E}$ are also
Hermitian and skew-Hermitian, respectively, \ie that $s\E-\A$ stays an even pencil when perturbed.

If we start to perturb a matrix, then its inertia remains constant in an open
neighborhood of the matrix only if it has no zero eigenvalues. Otherwise the inertia will change for arbitrarily small perturbations, unless we impose constraints on the type of perturbations that are allowed. Therefore we will need to impose that our perturbation  preserves the rank of the matrix $\E$.

When computing the eigenstructure of even pencils such as (\ref{system}) or  (\ref{systemeq}), then there exist algorithms that  guarantee these properties, see \cite{BenBMX02,BenLMV13} and the references therein. We will employ the even implicitly restarted Arnoldi method of \cite{MehSS12},  in which  $\A+\Delta_{\mathcal A}$ stays Hermitian, $\E+\Delta_{\mathcal E}$ stays skew-Hermitian, and the null-space of $\E+\Delta_{\mathcal E}$ is preserved.
When our perturbation results from an eigenvalue algorithm, we can therefore assume that the perturbation $s\Delta_{\mathcal E}-\Delta_{\mathcal A}$ of the pencil $s\E-\A$ satisfies
\begin{equation} \label{struct1}
\Delta_{\mathcal E}= \left[ \begin{array}{ccc} \Delta^{\mathcal E}_{11} & \Delta^{\mathcal E}_{12} & 0 \\ \Delta^{\mathcal E}_{21} & \Delta^{\mathcal E}_{22} & 0 \\ 0 & 0 & 0 \end{array} \right] = -\Delta_{\mathcal E}^\mathsf{H}, \quad \Delta_{\mathcal A}=
\left[ \begin{array}{ccc} \Delta^{\mathcal A}_{11} & \Delta^{\mathcal A}_{12} & \Delta^{\mathcal A}_{13} \\ \Delta^{\mathcal A}_{21} & \Delta^{\mathcal A}_{22} & \Delta^{\mathcal A}_{23} \\ \Delta^{\mathcal A}_{31} & \Delta^{\mathcal A}_{32} & \Delta^{\mathcal A}_{33} \end{array} \right] =  \Delta_{\mathcal A}^\mathsf{H}.
\end{equation}

If the perturbation arises from an approximation of the model (such as discretization or model reduction), then this approximation process needs to be done in such a way, that constraints (such as Kirchhoff's conditions in networks, or position constraints as in mechanical systems) that result from the physical properties of the system are not destroyed, see \cite{BeaGM19}. If this is done correctly then again the structure (\ref{strut1}) is typically preserved.

\subsection{Bounds on the structured backward errors}\label{sec:errorbounds}

If we use a backward stable algorithm structure preserving algorithm from \cite{MehSS12} to compute the eigenstructure, then $\| (\Delta_{\mathcal E},\Delta_{\mathcal A})\|_F \approx \epsilon \|({\mathcal E},{\mathcal A})\|_F$ where $\epsilon$ is a small multiple of the machine precision (unit round-off), and   $\Delta_{\mathcal A}$ and $\Delta_{\mathcal E}$ have the structure indicated in \eqref{struct1}.  To see whether the computed eigenstructure is that associated with a pH descriptor system and to compute the backward error, we need  to find a transformation that preserves the computed eigenstructure, preserves the structure
indicated in \eqref{struct1},  annihilates the diagonal blocks
$s\Delta^{\mathcal E}_{11}-\Delta^{\mathcal A}_{11}$ and $s\Delta^{\mathcal E}_{22}-\Delta^{\mathcal A}_{22}$, and also restores the property  $E+\Delta^{\mathcal E}_{12}=(E+\Delta^{\mathcal E}_{12})^\mathsf{H}> 0$.

To preserve the computed eigenstructure and  the Hermitian character of
$\Delta_{\mathcal A}$, we will perform a congruence transformation;
and in order to preserve the structure of $\Delta_\E$ indicated in \eqref{struct1},
we will constrain it to be block lower triangular, \ie we look for a transformation
\[
  Z := \left[ \begin{array}{ccc} Z_{11} & Z_{12} & 0 \\ Z_{21} & Z_{22} & 0 \\
Z_{31} & Z_{32} & Z_{33} \end{array} \right]
\]
such that
\[
 Z^\mathsf{H}\left(\A+\Delta_{\mathcal A}\right)Z=
\left[ \begin{array}{ccc} 0 & A+\Delta_A & B+\Delta_B \\ A^{\mathsf{H}}+\Delta_A^{\mathsf{H}} & 0 & C^\mathsf{H}+\Delta_C^{\mathsf{H}} \\ B^\mathsf{H}+\Delta_B^{\mathsf{H}} & C+\Delta_C & D^\mathsf{H}+\Delta_D^{\mathsf{H}}+D+\Delta_D \end{array} \right],
\]
and
\[
Z^\mathsf{H}\left(\E+\Delta_{\mathcal E}\right)Z=
\left[ \begin{array}{ccc} 0 & E+\Delta_E & 0 \\ -E^{\mathsf{H}}-\Delta_E^{\mathsf{H}} & 0 & 0 \\
0 & 0 & 0 \end{array} \right],
\]
with $(E+\Delta_E)^{\mathsf{H}}=E+\Delta_E > 0$.
We also require that $Z$ is as close as possible to the identity matrix, such that
$\{\Delta_E,\Delta_A,\Delta_B,\Delta_C,\Delta_D\}$ remain as small as possible.
This suggests that we choose $Z_{31}=Z_{32}=0$ and $Z_{33}=I_m$ and
look for a submatrix of $Z$
\[
 \left[ \begin{array}{ccc} Z_{11} & Z_{12} \\ Z_{21} & Z_{22}  \end{array} \right] := I_{2n} + Y = I_{2n} + \left[ \begin{array}{ccc} Y_{11} & Y_{12} \\ Y_{21} & Y_{22}  \end{array} \right]
\]
near the identity matrix, and satisfying the matrix equations
\begin{eqnarray} \label{A}
(I+Y^\mathsf{H})\left[ \begin{array}{ccc} \Delta^{\mathcal A}_{11} & A+\Delta^{\mathcal A}_{12} \\ A^{\mathsf{H}}+\Delta^{\mathcal A}_{21} & \Delta^{\mathcal A}_{22} \end{array} \right] (I+Y) &= & \left[ \begin{array}{ccc} 0 & A+\Delta_A \\ A^{\mathsf{H}}+\Delta_A^{\mathsf{H}} & 0 \end{array} \right],
\\ \label{E}
(I+Y^\mathsf{H})\left[ \begin{array}{ccc} \Delta^{\mathcal E}_{11} & E+\Delta^{\mathcal E}_{12} \\-E^{\mathsf{H}}+\Delta^{\mathcal E}_{21} & \Delta^{\mathcal E}_{22} \end{array} \right] (I+Y)  & = & \left[ \begin{array}{ccc} 0 & E+\Delta_E \\ -E^{\mathsf{H}}-\Delta_E^{\mathsf{H}} & 0 \end{array} \right].
\end{eqnarray}
Removing common terms on both sides and using the notation $A_\Delta:=A+\Delta^\A_{12}$ and $E_\Delta:=E+\Delta^\E_{12}$, we can rewrite these equations as
\[
 Y^\mathsf{H}\left[ \begin{array}{cc} \Delta^{\mathcal A}_{11} & A_\Delta \\ A_\Delta^{\mathsf{H}} & \Delta^{\mathcal A}_{22}  \end{array} \right] +\left[ \begin{array}{cc} \Delta^{\mathcal A}_{11} & A_\Delta \\ A_\Delta^{\mathsf{H}} & \Delta^{\mathcal A}_{22}  \end{array} \right]Y +\left[ \begin{array}{ccc} \Delta^{\mathcal A}_{11} & \Delta^{\mathcal A}_{12} \\ \Delta^{\mathcal A}_{21} & \Delta^{\mathcal A}_{22} \end{array} \right] =  \left[ \begin{array}{ccc} 0 & \Delta_A \\ \Delta_A^{\mathsf{H}} & 0 \end{array} \right]-Y^\mathsf{H} \left[ \begin{array}{ccc} \Delta^{\mathcal A}_{11} & A_\Delta  \\ A_\Delta^\mathsf{H} & \Delta^{\mathcal A}_{22} \end{array} \right] Y,
\]
and
\[
Y^\mathsf{H}\left[ \begin{array}{cc} \Delta^{\mathcal E}_{11} & E_\Delta \\ -E_\Delta^{\mathsf{H}} & \Delta^{\mathcal E}_{22}  \end{array} \right] +\left[ \begin{array}{cc} \Delta^{\mathcal E}_{11} & E_\Delta \\ -E_\Delta^{\mathsf{H}} & \Delta^{\mathcal E}_{22}  \end{array} \right]Y +\left[ \begin{array}{ccc} \Delta^{\mathcal E}_{11} & \Delta^{\mathcal E}_{12} \\ \Delta^{\mathcal E}_{21} & \Delta^{\mathcal E}_{22} \end{array} \right] =  \left[ \begin{array}{ccc} 0 & \Delta_E \\ -\Delta_E^{\mathsf{H}} & 0 \end{array} \right]
-Y^\mathsf{H}\left[ \begin{array}{ccc} \Delta^{\mathcal E}_{11} & E_\Delta \\ -E_\Delta^\mathsf{H} & \Delta^{\mathcal E}_{22} \end{array} \right]Y,
\]
in which we need to zero out the diagonal blocks. Considering these equations, it seems  reasonable to choose $Y_{11}=Y_{22}=0$ and then solve the remaining quadratic equations
\begin{eqnarray}
\label{A1}  A_\Delta Y_{21} + Y_{21}^\mathsf{H}A_\Delta^\mathsf{H}  & = & - \Delta^\A_{11}-Y_{21}^\mathsf{H}\Delta^\A_{22}Y_{21}, \label{10}\\
\label{E1}  E_\Delta Y_{21} - Y_{21}^\mathsf{H} E_\Delta^\mathsf{H}  & = & - \Delta^\E_{11}-Y_{21}^\mathsf{H}\Delta^\E_{22}Y_{21}, \label{11}\\
\label{A2}  A_\Delta^\mathsf{H} Y_{12} + Y_{12}^\mathsf{H} A_\Delta  & = & -\Delta^\A_{22}-Y_{12}^\mathsf{H}\Delta^\A_{11}Y_{12}, \label{12}\\
\label{E2}  - E_\Delta^\mathsf{H} Y_{12} + Y_{12}^\mathsf{H} E_\Delta  & = & -\Delta^\E_{22}-Y_{12}^\mathsf{H}\Delta^\E_{11}Y_{12}\label{13}
\end{eqnarray}
for the unknowns $Y_{12}$ and $Y_{21}$.
If we decompose $Y_{12}$ and $Y_{21}$ in their Hermitian and skew Hermitian parts,
$Y_{12}=W_{12}+V_{12}$, and $Y_{21}=W_{21}+V_{21}$, with
$ W_{12}= W_{12}^\mathsf{H}$,  $W_{21}= W_{21}^\mathsf{H}$,  $V_{12}=-V_{12}^\mathsf{H}$, and $V_{21}=-V_{21}^\mathsf{H}$,
then, using the $\vect$ function which stacks the columns of a matrix in a large vector, we have
\begin{eqnarray*} \label{vec12} \vect(Y_{12})&=&\vect(W_{12})+\vect(V_{12}),  \quad \vect(Y^\mathsf{H}_{12})=\vect(W_{12})-\vect(V_{12}),\\ \label{vec21} \vect(Y_{21})&=&\vect(W_{21})+\vect(V_{21}),  \quad \vect(Y^\mathsf{H}_{21})=\vect(W_{21})-\vect(V_{21}).
\end{eqnarray*}
We can then rewrite the equations (\ref{10})--(\ref{13}) using Kronecker products  as
\begin{eqnarray} \label{L1}
	\left[\begin{array}{cr}
	I_n\otimes E_\Delta & 	- \overline E_\Delta\otimes I_n \\ I_n\otimes A_\Delta & \overline A_\Delta\otimes I_n
\end{array} \right]
\left[\begin{array}{rr}
	\vect(W_{21}) + \vect(V_{21}) \\ \vect(W_{21})  - \vect(V_{21})
\end{array} \right] =  \left[\begin{array}{r}
	- \vect(\Delta^\E_{11}) \\ - \vect(\Delta^\A_{11})
\end{array} \right]  + {\mathcal O}(\|Y\|^2), \\  \label{L2}
	\left[\begin{array}{cr}
	-I_n\otimes E_\Delta^\mathsf{H} & E_\Delta^\mathsf{T}\otimes I_n \\ I_n\otimes A_\Delta^\mathsf{H} & A_\Delta^\mathsf{T}\otimes I_n
\end{array} \right]
\left[\begin{array}{rr}
	\vect(W_{12}) + \vect(V_{12}) \\ \vect(W_{12}) - \vect(V_{12})
\end{array} \right] = \left[\begin{array}{r}
 -\vect(\Delta^\E_{22})  \\  -\vect(\Delta^\A_{22})
\end{array} \right]+ {\mathcal O}(\|Y\|^2).
\end{eqnarray}
If we ignore the quadratic terms on the right hand side, then we obtain linear systems that are solvable when the pencils $sE_\Delta-A_\Delta$ and $-sE_\Delta^\mathsf{H}-A_\Delta^\mathsf{H}$ have no common eigenvalues, see \eg \cite{LanT85}, which is the case when these pencils come from a sufficiently small perturbation of a system $\T(s)$ which is strictly passive. We have the following result.
\begin{lemma} \label{lembound}
Consider the linear systems \eqref{L1}--\eqref{L2} with the quadratic terms set to $0$, set
\begin{equation} \label{K1K2}
K_1(E,A):= \left[\begin{array}{cr}
I_n\otimes E & 	- \overline E\otimes I_n \\ I_n\otimes A & \overline A\otimes I_n
\end{array} \right], \quad K_2(E,A):=	\left[\begin{array}{cr}
-I_n\otimes E^\mathsf{H} & E^\mathsf{T}\otimes I_n \\ I_n\otimes A^\mathsf{H} & A^\mathsf{T}\otimes I_n
\end{array} \right],
\end{equation}
and let
\[
\hat\delta:=\max(\|\Delta_{12}^\A\|_2,\|\Delta_{12}^\E\|_2) < \frac12\min\{\sigma_{2n^2}(K_1(E,A)),\sigma_{2n^2}(K_2(E,A))\},
\]
where $\sigma_j(M)$ denotes the $j$th singular value of the matrix $M$. Then the solution $(Y_{21}, Y_{12})$ satisfies the bound
\begin{equation} \label{linbound} \sqrt{2} \| ( Y_{21}, Y_{12} ) \|_F \le
\frac{\| ( \Delta^\E_{11},\Delta^\A_{11},\Delta^\E_{22},\Delta^\A_{22} ) \|_F}{\min\{\sigma_{2n^2}(K_1(E,A)),\sigma_{2n^2}(K_2(E,A))\}-2\hat\delta}.
\end{equation}
\end{lemma}
\begin{proof}
Define $\hat K_i:= K_i(E_\Delta,A_\Delta)= K_i(E+\Delta_{12}^\E,A+\Delta_{12}^\A)$ for $i=1,2$, then it follows from standard perturbation theory, see \eg  \cite{Hig02}, that
\[
 \sigma_{2n^2}(K_1(E_\Delta,A_\Delta))\ge \sigma_{2n^2}(K_1(E,A))-2\hat\delta ,  \quad \sigma_{2n^2}(K_2(E_\Delta,A_\Delta)) \ge \sigma_{2n^2}(K_2(E,A))-2\hat \delta .
\]
The bound \eqref{linbound} then follows from the solutions of the linear systems  \eqref{L1}--\eqref{L2}, which can be written as
%
\begin{equation} \label{consistent}
\begin{array}{c}
\sqrt{2} X \left[ \begin{array}{c} \vect(W_{21})  \\ \vect(V_{21})  \end{array} \right] =
 \left[\begin{array}{r} \vect(Y_{21})  \\ \vect(Y^\mathsf{H}_{21})  \end{array} \right]
 = -  \hat K_1^{-1}  \left[ \begin{array}{r} \vect(\Delta^\E_{11}) \\ \vect(\Delta^\A_{11}) \end{array} \right] , \\ [+5mm]
\sqrt{2} X \left[ \begin{array}{c} \vect(W_{12})  \\ \vect(V_{12})  \end{array} \right]  = \left[\begin{array}{r}
	\vect(Y_{12})  \\ \vect(Y^\mathsf{H}_{12})  \end{array} \right] = -  \hat K_2^{-1}
 \left[\begin{array}{r} \vect(\Delta^\E_{22}) \\ \vect(\Delta^\A_{22})  \end{array} \right] ,  \end{array}
\end{equation}
and the fact that $\|\vect(M)\|_2=\|M\|_F$ for any matrix $M$.
\end{proof}
An  estimate of the smallest singular values $\sigma_{2n^2}(K_1(E,A))$ and $\sigma_{2n^2}(K_2(E,A))$ is obtained from considering the triple $(\sigma_n,u_n,v_n)$ in \eqref{stabpert} which yields
\begin{eqnarray*}
\label{BK1}
	\left[\begin{array}{cc} \imath\omega(u_n^\mathsf{T}\otimes u_n^\mathsf{H})
		& -(u_n^\mathsf{T}\otimes u_n^\mathsf{H}) \end{array}  \right]
	\left[\begin{array}{cr}
		I_n\otimes E & 	- \overline E\otimes I_n \\ I_n\otimes A & \overline A\otimes I_n
	\end{array} \right] & = & \sigma_n	\left[\begin{array}{cc} (u_n^\mathsf{T}\otimes v_n^\mathsf{H})
		& (v_n^\mathsf{T}\otimes u_n^\mathsf{H}) \end{array}  \right], \\
\label{BK2}
		\left[\begin{array}{cc} \imath\omega(v_n^\mathsf{T}\otimes v_n^\mathsf{H})
		& -(v_n^\mathsf{T}\otimes v_n^\mathsf{H}) \end{array}  \right]
	\left[\begin{array}{cr}
		- I_n\otimes E^\mathsf{H} &  E^\mathsf{T}\otimes I_n \\ I_n\otimes A^\mathsf{H} & A^\mathsf{T}\otimes I_n
	\end{array} \right] & = & \sigma_n	\left[\begin{array}{cc} (v_n^\mathsf{T}\otimes u_n^\mathsf{H})
		& (u_n^\mathsf{T}\otimes v_n^\mathsf{H}) \end{array}  \right].
\end{eqnarray*}
It follows from these identities that the smallest singular value of the unperturbed
block Kronecker products must be smaller or equal to $\sqrt{2}\sigma_n/\sqrt{1+\omega^2} = \sqrt{2}\rho(E,A)$.
Defining the  smallest structured singular value of a structured matrix  as the smallest structured perturbation that makes it singular, one can expect that this is a very good estimate, since the smallest structured singular value equals the stability radius $\rho(E,A)$. The quality of this estimate is illustrated via numerical examples in Section~\ref{sec:numerical}.

\subsection{An iteration solution procedure}\label{sec:iterative}
The solution of the quadratic equations (\ref{A1})--(\ref{E1}) in $(Y_{21},Y_{21}^\mathsf{H})$ and (\ref{A2})--(\ref{E2}) in $(Y_{12},Y_{12}^\mathsf{H})$,
can be obtained using the  iterative schemes~
\[
 \begin{array}{c}
A_\Delta [Y_{21}]_{i+1} + [Y_{21}^\mathsf{H}]_{i+1} A_\Delta^\mathsf{H}   =  - \Delta^\A_{11}-[Y_{21}^\mathsf{H}]_i\Delta^\A_{22}[Y_{21}]_i, \\
E_\Delta [Y_{21}]_{i+1} - [Y_{21}^\mathsf{H}]_{i+1} E_\Delta^\mathsf{H}   = - \Delta^\E_{11}-[Y_{21}^\mathsf{H}]_i\Delta^\E_{22}[Y_{21}]_i, \end{array}
\]
and
\[
\begin{array}{c}
A_\Delta^\mathsf{H} [Y_{12}]_{i+1} + [Y_{12}^\mathsf{H}]_{i+1} A_\Delta   =  -\Delta^\A_{22}-[Y_{12}^\mathsf{H}]_i\Delta^\A_{11}[Y_{12}]_i, \\
- E_\Delta^\mathsf{H} [Y_{12}]_{i+1} + [Y_{12}^\mathsf{H}]_{i+1} E_\Delta  =  -\Delta^\E_{22}-[Y_{12}^\mathsf{H}]_i\Delta^\E_{11}[Y_{12}]_i.\end{array}
\]
Using an analysis similar to that of \cite{SteS90}, we can show that these iterations converge to a solution of the quadratic equations \eqref{A1}, \eqref{E1}, \eqref{A2}, and \eqref{E2}, see \cite[Theorem 2.11, p. 242]{SteS90} and \cite{DopLPV18}. We obtain the following main result.
\begin{theorem}\label{thm:gen_sylvester_solution}
Consider the system of matrix equations  (\ref{A1}), (\ref{E1}), (\ref{A2}), (\ref{E2}).
Let
\begin{eqnarray*}
 \delta &:=& \min\{ \sigma_{2n^2} (K_1(E_\Delta,A_\Delta)),\sigma_{2n^2} (K_2(E_\Delta,A_\Delta)) \} -2 \max\{ \|\Delta_{12}^\A\|_2,\|\Delta_{12}^\E\|_2\},\\
 \theta &:=& \|(\Delta^\E_{11},\Delta^\E_{22},\Delta^\A_{11},\Delta^\A_{22})\|_F,\\
 \omega &:=& \sqrt{2}
\left\| ( \Delta^{\mathcal A}_{11} ,\Delta^{\mathcal A}_{22}, \Delta^{\mathcal E}_{11}, \Delta^{\mathcal E}_{22} )  \right\|_F.
\end{eqnarray*}
If  $\delta>0$ and $ \frac{\theta\omega}{\delta^2}<\frac{1}{4}\>$, then there exists a solution $(Y_{12},Y_{21})$ of these equations satisfying
\begin{equation}\label{eq:norm_solution}
  \|(Y_{12},Y_{21})\|_F\leq 2\theta/\delta.
  \end{equation}
\end{theorem}
\begin{proof}
Lemma \ref{lembound} and the assumption $\delta >0$ guarantee that the linear system of matrix equations \eqref{consistent}
is solvable. If we write its solution $([Y_{21}]_1,[Y_{12}]_1)$ in terms of the representation with the matrices $(W_{21},V_{21})$ for
$Y_{21}$ and with $(W_{12},V_{12})$ for $Y_{12}$, then we obtain the bound
\[
 \|([W_{21}]_1,[V_{21}]_1,[W_{12}]_1,[V_{12}]_1)\|_F\leq \frac{\|(\Delta^\E_{11},\Delta^\E_{22},\Delta^\A_{11},\Delta^\A_{22})\|_F}{\delta} = \frac{\theta}{\delta}=:\rho_0,
\]
using Lemma \ref{lembound}.
 The iterative schemes can then be written as
 \begin{eqnarray} \nonumber
 \left[ \begin{array}{c} \vect([W_{21}]_{i+1})  \\ \vect([V_{21}]_{i+1})  \end{array} \right] &=&
 \left[ \begin{array}{c} \vect([W_{21}]_0)  \\ \vect([V_{21}]_0)  \end{array} \right] \\
 \label{it1}  && + (\hat K_1\sqrt{2} X)^{-1}
 \left[ \begin{array}{c} \vect([W_{21}-V_{21}]_i \Delta^\E_{11}[W_{21}+V_{21}]_i)  \\
 \vect([W_{21}-V_{21}]_i \Delta^\A_{11}[W_{21}+V_{21}]_i)   \end{array} \right],\\
\nonumber
 \left[ \begin{array}{c} \vect([W_{12}]_{i+1})  \\ \vect([V_{12}]_{i+1})  \end{array} \right] &=&
 \left[ \begin{array}{c} \vect([W_{12}]_0)  \\ \vect([V_{12}]_0)  \end{array} \right] \\
  \label{it2} &&+ (\hat K_2\sqrt{2} X)^{-1}
 \left[ \begin{array}{c} \vect([W_{12}-V_{12}]_i \Delta^\E_{22}[W_{12}+V_{12}]_i)  \\
 \vect([W_{12}-V_{12}]_i \Delta^\A_{22}[W_{12}+V_{12}]_i)   \end{array} \right].
\end{eqnarray}
We now show that the sequences $\{[W_{12},V_{12}]_i)\}_{i=0}^\infty$ and $\{[W_{21},V_{21}]_i)\}_{i=0}^\infty$ converge to a solution of  (\ref{A1}), (\ref{E1}),  (\ref{A2}), (\ref{E2}) satisfying \eqref{eq:norm_solution}.
To prove this, we first show that these sequences are bounded. The proofs for $\{[W_{12},V_{12}]_i)\}_{i=0}^\infty$ and $\{[W_{21},V_{21}]_i)\}_{i=0}^\infty$ are identical, so we only prove it for one sequence and we drop the indices of $W$, $V$, $K$,
$\Delta^\E$ and $\Delta^\A$,  in order to simplify the notation.
If $\|(W_{i},V_{i})\|_F \leq \rho_{i}$, then from (\ref{it1}) and (\ref{it2}) we have that
  \begin{align*}
    \|(W_{i+1},V_{i+1})\|_F\nonumber \leq  &\|(W_0,V_0)\|_F +\sqrt{2}\|\hat K^{-1} \|_2\|(W_{i},V_{i})\|_F^{2}\|(\Delta^\E,\Delta^\A)\|_F\nonumber\\
    \leq &\rho_0+\rho_{i}^2\omega\delta^{-1}=:\rho_{i+1}\>.
    \label{eq:recurrecne_rho_i_1}
  \end{align*}
We may write the quantity $\rho_{i}$ in this equation as $\rho_{i}=\rho_0(1+\kappa_i)$, where $\kappa_i$ satisfies the recursion
\begin{equation} \label{eq:fixedstewart}
  \left\{ \begin{array}{l}
    \kappa_1=\rho_0\omega\delta^{-1}=\theta \omega \delta^{-2},\\
    \kappa_{i+1}=\kappa_1(1+\kappa_i)^2\>.
    \end{array}\right.
\end{equation}
An induction argument used in \cite{DopLPV18} then shows that $0 < \kappa_1 < \kappa_2 < \cdots$, \ie that the sequence is strictly increasing and that, if $\kappa_1<1/4$, then
\[
\kappa =\lim_{i\rightarrow \infty} \kappa_i = \frac{2\kappa_1}{1-2\kappa_1+\sqrt{1-4\kappa_1}} < 1,
  \]
and $\kappa_i <\kappa$ for all $i\geq 1$.
Thus, the norms of the elements of the sequence $\{(W_i,V_i)\}_{i=0}^\infty$ are bounded as
  \begin{equation} \label{eq:bound_CD2}
    \|(W_{i},V_{i})\|_F\leq \rho :=\lim_{i\rightarrow \infty}{\rho_i}=\rho_0(1+\kappa)\>.
  \end{equation}
It is shown in \cite{DopLPV18} that the sequence $\{ (W_i,V_i)\}_{i=0}^\infty$ is a Cauchy sequence and therefore converges, provided that $2\delta^{-1}\omega \rho<1$, which is ensured by \eqref{eq:norm_solution}.
Finally, from \eqref{eq:bound_CD2}, $\|(W,V)\|_F \leq \rho_0(1+\kappa) < 2\rho_0 = 2\delta^{-1}\theta$, which concludes the proof.
\end{proof}

Once the zero blocks have been restored, we still need to restore the property that
$E$ was Hermitian and positive definite. This can be incorporated in the pencil via an additional congruence transformation
$Z=\diag(I_n, Z_{22},I_m)$, where $Z_{22}$ is the polar factor of the perturbed matrix $E+\Delta_E$.  It was shown in \cite{Sun89} that the polar factor of a perturbed positive definite Hermitian matrix $E+\Delta_E$
is near the identity matrix and if expressed as $  Z_{22}= I+Y_{22}$ satisfied the bound
\[
 \|Y_{22}\|_F \le 2 \|E^{-1}\|_2 \|\Delta_E\|_F.
\]
We can thus restore also the positive definite symmetry of the matrix $E$ at the cost of a growth factor $2 \|E^{-1}\|_2$ in the blocks $E_\Delta$, $A_\Delta$ and $C_\Delta$, since the congruence transformation yields a right multiplication of these matrices by $I_n+Y_{22}$. It is worth pointing out that $\|E^{-1}\|_2\le \frac{1}{\rho(E,A)}$, since the limit of $\sigma_n(A-\imath\omega E)/\sqrt{1+\omega^2}$ for increasing $\omega$ is $\sigma_n(E)$. The numerical errors corresponding to this second step are therefore of the same order of magnitude as in the first step. But this also shows that a very small stability radius $\rho$ gives very large backward errors.

\subsection{The complete procedure}\label{sec:complete}

The combination of the two steps in computing the structured perturbation described in the previous subsection corresponds to
a congruence transformation $Z$ of the form
\[
 Z = \left[ \begin{array}{cc|c} I_n & Y_{12} & 0 \\ Y_{21} & I_n & 0 \\ \hline
0 & 0 & I_m \end{array} \right]\left[ \begin{array}{cc|c} I_n & 0 & 0 \\ 0 & I_n+Y_{22} & 0 \\ \hline 0 & 0 & I_m \end{array} \right] = \left[ \begin{array}{c|c} I_{2n} + \hat Y \\ \hline
0  & I_m \end{array} \right],
\]
 where
 \[
 \|\hat Y\|_F \le 2 ( \|(\Delta^\E_{11},\Delta^\A_{11},\Delta^\E_{22},\Delta^\A_{22})\|_F +\|\Delta^\E_{12}\|_F ) /\delta + {\mathcal O}(\epsilon^2).
\]
Note that the zero blocks created in the first step are not destroyed in the second step and the error growth of the two stages just add together (except for the second order terms).

It follows that forcing the pH structure of the pencil \eqref{system} requires a growth of a factor $1/\rho(E,A)$ in the perturbations of the blocks $E_\Delta$, $A_\Delta$, $B_\Delta$ and $C_\Delta$, but not in $D_\Delta$.

If one wants to find the corresponding errors in the representation $R$, $J$, $G$ and $P$, we can use the linear transformation between the two representations which yields
\[
 \left[ \begin{array}{rr} - \Delta_R & \Delta_G \\ \Delta_J & -\Delta_P  \end{array} \right]=   \frac{1}{\sqrt{2}} X\left[ \begin{array}{ccc} \Delta_A & \Delta_B \\ \Delta_A^\mathsf{H} & \Delta_C^\mathsf{H}  \end{array} \right]
\]
which shows that the backward errors are of the same order of magnitude.

\begin{remark}\label{rem:2}{\rm
We remark that we did not attempt to preserve passivity; we only made sure that the pencil structure is preserved. But if the original perturbation $s\Delta_\E-\Delta_\A$ did not destroy passivity, then the restoration also does not destroy it, since it is a congruence
transformation on the pencil $\S(s)$. This follows from the discussion in the beginning of this section. 
}
\end{remark}

\subsection{Passivity restoration}\label{sec:passrestore}

As we have discussed in Remark~\ref{rem:2}, when the original perturbation did not destroy passivity then the procedure will
still deliver a passive system. However, in many applications the system starts out as a passive system model and then discretization or model reduction may destroy passivity. Whether this has happend can be observed by checking the
eigenvalues of the pencils eigenstructure of the pencils \eqref{system} (or  \eqref{systemeq}) with a structure preserving method. If this pencil  has purely imaginary eigenvalues or if the pencil is singular or has index greater than one, which can be checked
by computing the rank of $\A$ when projected to the kernel of $\E$, then the underlying system (\ref{gstatespace}) is not passive any longer. In this case it has been a difficult and essentially still an open problem to find the smallest perturbation to the system matrices  $\{ E,A,B,C,D\}$ in order to restore passivity.
One would hope that this requires a correction on the order of the perturbation that has been already comitted; see \cite{AlaBKMM10,Bru11,BruS12,FreJ04,FreJV07,Gri03,Gri04,GusS01,OveV05}, mostly for the case of standard state space systems.

For descriptor systems this question was mostly open, but our procedure  from the last subsection suggest an immediate solution to the problem. We can first perturb the pencil \eqref{system} (or  \eqref{systemeq}) so that it does not have purely imaginary eigenvalues anymore, actually one wants to produce a reasonable margin arround the imaginary axis, where there should be no eigenvalues, see the procedures in \cite{AlaBKMM10,BruS12} for the standard case.

But before one can use these procedures one needs a perturbation that fixes the pencil to be regular and of index one. This can be done as follows.
If the matrix $E$ is not already in partitioned form
\begin{equation}\label{pf}
E=\left [\begin{array}{cc} E_{11} & 0 \\ 0 & 0 \end{array}\right ]
\end{equation}
with $E_{11}$ positive definite, then one can achieve this via a spectral decomposition, Cholesky factorization, or  singular value decomposition of $E\geq 0$. However in many applications this partitioning already exists: see \cite{BeaGM20} for a canonical form, where the structure is discussed in the pH descriptor form or \cite{BunBMN99} for the general case.

Let us therefore assume that $E$ has the partitioned form \eqref{pf} and partition  $A,B,C$ conformally as
\[
A=\left [\begin{array}{cc} A_{11} & A_{12} \\ A_{21} & A_{22} \end{array}\right ],
\ B=\left [\begin{array}{c} B_1 \\ B_2 \end{array}\right ], \ C=\left [\begin{array}{cc} C_1 & C_2 \end{array}\right ].
\]
Then the system \eqref{system} is index one if and only if  the matrix
\[
\hat S :=\left [ \begin{array} {ccc} 0 & A_{22} & B_2 \\ A_{22}^\mathsf{H} & 0 &  C_2^\mathsf{H}   \\ B_1^\mathsf{H} & C_2 &D+D^\mathsf{H} \end{array} \right ],
\]
is invertible. For passivity we also need that $D+D^\mathsf{H}>0$. So we should at least perturb the pencil so that
$D+\Delta_D+(D+\Delta_D)^\mathsf{H} >\mu I$, where $\mu$ is the perturbation level (from discretization or model reduction) that has led to the system matrices $\{E,A,B,C,D\}$. This can be easily done by taking  $\Delta_D=\frac{\mu}2I$. If the so perturbed pencil $s(\E+\Delta_{\mathcal E}) -(\A +\Delta_{\mathcal A})$ is not index one, then further perturbations are necessary. If the original pencil $sE-A$ is index one, \ie $A_{22}$ is invertible or if the matrices $[A_{22}, B_2]$
and $[A_{22}^{\mathsf{H}}, C_2^{\mathsf{H}}]$ have full rank, \ie the system is controllable and observable at $\infty$
then one can increase $\mu$ further until $\hat S$ is invertible or remove the uncontrollable part, see \cite{BunBMN99}.

\section{Numerical results}\label{sec:numerical}
In this section we describe numerical experiments illustrating the results of the previous section.
The numerical tests were carried out in Matlab version R2019a running on an Intel Core i5 processor, with machine precision $\epsilon=2.2204e-16$. In this first test, we generated a passive system $\{A,B,C,D,E\}$ with a stability radius for the pencil $sE-A$ of the order of $0.5$. The stability radius was computed  with 5 digits of accuracy as $\rho(E,A)= 4.0537e-01$. We then perturbed the structured pencil ${\mathcal S}(s)$ in  \eqref{system}
with a random perturbation of the form \eqref{struct1}
and of approximate norms $\delta_i=10^{-i}$ for $i=1,2, \ldots, 10$, and applied the iterative procedure of
Section~\ref{sec:complete}. We report in Table \ref{tab:table1} the quantities
$ \delta(E,A):= \| ( \Delta^{\mathcal A}_{11},\Delta^{\mathcal A}_{22},\Delta^{\mathcal E}_{11},\Delta^{\mathcal E}_{22} ) \|_F$
as a function of the number of iterations needed to reach convergence. The first column (for $k=0$) corresponds to the initial perturbations of the order
of $\delta_i=10^{(-i)}$. The next columns indicate the convergence behaviour, which is at least quadratic (and possibly cubic).

\begin{table}[!ht]
  \begin{center}
    \caption{Evolution of $\delta_k(E,A)$ as function of the number of iterations}
    \label{tab:table1}
    \begin{tabular}{c|cccc} 
 $\delta$ &  $\delta_0(E,A)$ & $\delta_1(E,A)$ &   $\delta_2(E,A)$ & $\delta_3(E,A)$  \\   \hline
1.e-01 &  5.3979e-01  & 3.1042e-02  & 1.3663e-07  & 3.2135e-16 \\
1.e-02 &  8.1884e-02  & 1.0068e-04  & 3.3695e-15  & 2.8047e-17 \\
1.e-03 &  6.4600e-03  & 3.8518e-08  & 4.1879e-18   \\
1.e-04 &  6.4211e-04  & 5.3011e-11  & 3.4952e-19    \\
1.e-05 &  5.2796e-05  & 4.5678e-14  & 7.2183e-20    \\
1.e-06 &  6.6959e-06  & 8.9386e-17     \\
1.e-07 &  6.2294e-07  & 3.5734e-20     \\
1.e-08 &  7.8891e-08  & 1.8682e-22     \\
1.e-09 &  7.8703e-09  & 1.6507e-23     \\
1.e-10 &  8.3537e-10  & 1.1050e-24     \\
    \end{tabular}
  \end{center}
\end{table}
In the second table, we look at how close the transformation $Z=I+Y$ that is restoring the structure of the pencil, is to identity, by comparing
$\|Y\|_F= \|Z-I_{2n+m}\|_F$ and $\delta$, the initial unstructured perturbation. Clearly, they are of the same order, indicating that
the restoration is very reasonable and of the same order as the original perturbation, provided the stability radius is not too small.
The third column gives the ratio $\delta_1(E,A)/[\delta_0(E,A)]^2$ for the structured error in the first iteration, which suggests that the process is at least quadratically convergent, and probably cubically convergent (which is often the case in Hermitian eigenvalue problems).
The last column is a verification of the bound in Lemma \ref{lembound}. The fact that the quantities $\sqrt{2}\|Y\|_F\rho(E,A)/\delta_0(E,A)$ are close to 1
indicates that the condition numbers of the matrices $K_1$ and $K_2$ are close to $\rho(E,A)$.

\begin{table}[!ht]
  \begin{center}
    \caption{Convergence rate and condition estimate}
    \label{tab:table2}
    \begin{tabular}{c|ccc} 
   $\delta$ & $\|Y\|_F$ & $\frac{\delta_1(E,A)}{[\delta_0(E,A)]^2}$ &  $\frac{\sqrt{2}\|Y\|_F\rho(E,A)}{\delta_0(E,A)}$ \\ \hline
   1.e-01 &  6.2266e-01  &  1.0654e-01 &  6.6129e-01 \\
   1.e-02 &  7.9702e-02  &  1.5016e-02 &  5.5800e-01 \\
   1.e-03 &  6.8401e-03  &  9.2299e-04 &  6.0701e-01 \\
   1.e-04 &  6.8668e-04  &  1.2857e-04 &  6.1306e-01 \\
   1.e-05 &  7.6133e-05  &  1.6387e-05 &  8.2668e-01 \\
   1.e-06 &  9.3534e-06  &  1.9937e-06 &  8.0080e-01 \\
   1.e-07 &  7.1741e-07  &  9.2085e-08 &  6.6021e-01 \\
   1.e-08 &  9.0728e-08  &  3.0017e-08 &  6.5929e-01 \\
   1.e-09 &  7.4229e-09  &  2.6649e-07 &  5.4069e-01 \\
   1.e-10 &  6.9376e-10  &  1.5835e-06 &  4.7610e-01 \\
    \end{tabular}
  \end{center}
\end{table}

In Table \ref{tab:table3} we look at the effect of the stability radius on the restoration results. We modified the previous model in order to have a stability radius that
is arbitrarily small, but yet larger than the perturbations added to the pencil. We used initial perturbations of the order of $\delta=1.e-10$ and let the
stability radius $\rho(E,A)$ vary between $1.e-1$ and $1.e-6$. The second column shows that the transformation $Z:=I_{2n+m}+Y$ starts to diverge from
the identity, but one can see from the next columns that one iteration step is enough to restore the original structure, and this for perturbations of the order of $\delta=1.e-10$! The last column indicates again that the stability radius $\rho(E,A)$ is a very good estimate of the conditioning of the restoration matrices $K_1$ and $K_2$.

\begin{table}[!ht]
  \begin{center}
    \caption{Effect of the stability radius on the convergence}
    \label{tab:table3}
    \begin{tabular}{ccccc} 
   $\rho(E,A)$ & $\|Y\|_F$ & $\delta_0(E,A)$ & $\delta_1(E,A)$ &  $\frac{\sqrt{2}\|Y\|_F\rho(E,A)}{\delta_0(E,A)}$ \\ \hline
   9.9594e-02 &  1.2223e-09 &  7.6689e-10 &  3.3150e-24 &  2.2449e-01 \\
   3.7125e-02 &  1.4023e-09 &  6.9962e-10 &  1.9251e-24 &  1.0523e-01 \\
   1.3499e-02 &  1.4400e-08 &  5.7975e-10 &  2.6391e-23 &  4.7417e-01 \\
   4.8692e-03 &  4.6806e-09 &  5.0612e-10 &  1.3357e-23 &  6.3683e-02 \\
   1.7506e-03 &  2.7512e-08 &  5.9770e-10 &  4.6125e-23 &  1.1395e-01 \\
   6.2759e-04 &  1.1548e-07 &  7.2317e-10 &  1.8410e-22 &  1.4173e-01 \\
   2.2373e-04 &  1.6259e-07 &  6.6532e-10 &  3.2569e-22 &  7.7321e-02 \\
   7.8553e-05 &  1.5850e-06 &  5.1837e-10 &  3.7833e-21 &  3.3968e-01 \\
   2.6375e-05 &  6.6021e-06 &  6.3500e-10 &  1.0618e-20 &  3.8780e-01 \\
   7.6220e-06 &  1.7707e-05 &  6.7353e-10 &  6.2209e-20 &  2.8338e-01 \\
    \end{tabular}
  \end{center}
\end{table}

It should be pointed out that passive systems are also stable and that in practice their stability radius is never so close to 0 as in the above example. In fact, when forcing the stability radius to be so small, we meanwhile lost the property of passivity in this example.

\section{Concluding remarks} \label{sec:conclusion}
When computing the eigenstructure of even matrix pencils associated with port-Hamiltonian descriptor system using a structured generalized eigenvalue method, one can expect
to loose the special structure present in the corresponding even pencil. This structure is also responsible for the special symmetry that is present in the computed spectrum.
But the question remains whether the computed spectrum actually corresponds to a nearby passive system.
In this paper, we showed that this is indeed the case, provided the perturbations satisfy some reasonable bounds. The construction of the nearby port-Hamiltonian system that corresponds exactly to the computed spectrum, was obtained by a congruence transformation that is very near the identity matrix.
We have performed a backward error analysis and shown that the departure from the identity is of the same order as the numerical errors induced by the eigenvalue solver, except for a moderate growth factor that depends on the stability radius of the poles of the system.

This procedure can also be applied to any para-Hermitian function $\Phi(s):=\T(s)+\T^\mathsf{H}(-s)$, as long as the transfer function
$\T(s)$ is stable, but not passive. We also show how to possibly exploit the ideas developed in this paper, in order to address the more challenging
problem of restoring the passivity of a perturbed system, that lost its passivity because of a perturbation.

\end{document}